\documentclass[11pt, oneside]{amsart}
\usepackage[text={5.58in,8.5in},centering,letterpaper,dvips]{geometry}
\usepackage[dvipsnames]{xcolor}
\usepackage{graphicx}
\usepackage{subcaption}
\usepackage{amsfonts}
\usepackage{epsf}
\usepackage{amssymb}
\usepackage{amsmath}
\usepackage{amscd}
\usepackage{tikz}
\usepackage{pdfpages}
\usepackage{fancyhdr}
\usepackage{setspace}
\usepackage{hyperref}
\usepackage[all]{xy}
\usetikzlibrary{matrix}
\usepackage{verbatim}
\usepackage{enumerate}

\theoremstyle{theorem}
\newtheorem{theorem}{Theorem}[section]
\newtheorem{proposition}[theorem]{Proposition}
\newtheorem{lemma}[theorem]{Lemma}
\newtheorem{question}[theorem]{Question}
\newtheorem{corollary}[theorem]{Corollary}

\theoremstyle{definition}

\newtheorem{remark}[theorem]{Remark}

\newcommand{\F}{\mathcal{F}}

\newcommand{\RP}{\mathbb{RP}}

\newcommand{\D}{\mathcal D}

\newcommand{\Kk}{\mathcal K}

\newcommand{\Tt}{\mathcal T}

\newcommand{\pd}{\partial}

\newcommand{\emp}{\emptyset}
\newcommand{\X}{\times}

\newcommand{\be}{\begin{enumerate}}
\newcommand{\ee}{\end{enumerate}}

\newcommand{\Pau}{\mathcal{P}}
\newcommand{\K}{\mathcal{K}}
\newcommand{\T}{\mathcal{T}}

\makeatletter
\def\@seccntformat#1{%
  \protect\textup{\protect\@secnumfont
    \ifnum\pdfstrcmp{subsection}{#1}=0 \bfseries\fi% subsection # in \bfseries
    \csname the#1\endcsname
    \protect\@secnumpunct
  }%
}  
\makeatother

\theoremstyle{theorem}

\makeatletter
\newtheorem*{rep@theorem}{\rep@title}
\newcommand{\newreptheorem}[2]{%
\newenvironment{rep#1}[1]{%
 \def\rep@title{#2 \ref{##1}}%
 \begin{rep@theorem}}%
 {\end{rep@theorem}}}
\makeatother

\newreptheorem{theorem}{Theorem}
\newreptheorem{lemma}{Lemma}
\newreptheorem{question}{Question}
\newreptheorem{corollary}{Corollary}
\newreptheorem{proposition}{Proposition}

\begin{document}

\rhead{\thepage}
\lhead{\author}
\thispagestyle{empty}

%\tableofcontents
%\listoffigures

\raggedbottom
\pagenumbering{arabic}
\setcounter{section}{0}

%%%%%%%%%%%%%%%%%%%%%%%%%%%%%%%%%%%%%%%%%%%%%%%%%%%%%%%%
%%%%%%%%%%%%%%%%%%%%%%%%%%%%%%%%%%%%%%%%%%%%%%%%%%%%%%%%
%%%%%%%%%%%%%%%%%%%%%%%%%%%%%%%%%%%%%%%%%%%%%%%%%%%%%%%%

\title{Cubic graphs induced by bridge trisections}
%\date{\today}

\author{Jeffrey Meier}
\address{Department of Mathematics \\ Western Washington University, Bellingham, WA 98225}
\email{jeffrey.meier@wwu.edu}
\urladdr{\url{https://jeffreymeier.org}}

\author{Abigail Thompson}
\address{Department of Mathematics \\ University of California, Davis, Davis, CA 95616}
\email{thompson@math.ucdavis.edu}
\urladdr{\url{https://www.math.ucdavis.edu/~thompson/}}

\author{Alexander Zupan}
\address{Department of Mathematics, University of Nebraska-Lincoln, Lincoln, NE 68588}
\email{zupan@unl.edu}
\urladdr{\url{http://www.math.unl.edu/~azupan2}}

\begin{abstract}
	Every embedded surface $\mathcal{K}$ in the 4-sphere admits a bridge trisection, a decomposition of $(S^4,\mathcal{K})$ into three simple pieces.  In this case, the surface $\mathcal{K}$ is determined by an embedded 1-complex, called the \emph{1-skeleton} of the bridge trisection.  As an abstract graph, the 1-skeleton is a cubic graph $\Gamma$ that inherits a natural Tait coloring, a 3-coloring of the edge set of $\Gamma$ such that each vertex is incident to edges of all three colors.  In this paper, we reverse this association:  We prove that every Tait-colored cubic graph is isomorphic to the 1-skeleton of a bridge trisection corresponding to an unknotted surface.  When the surface is nonorientable, we show that such an embedding exists for every possible normal Euler number.  As a corollary, every tri-plane diagram for a knotted surface can be converted to a tri-plane diagram for an unknotted surface via crossing changes and interior Reidemeister moves.
\end{abstract}

\maketitle

%%%%%%%%%%%%%%%%%%%%%%%%%%%%%%%%%%%%%%%%%%%%%%%%%%%%%%%%%%
%%%%%%%%%%%%%%%%%%%%%%%%%%%%%%%%%%%%%%%%%%%%%%%%%%%%%%%%%%
\section{Introduction}\label{intro}
%%%%%%%%%%%%%%%%%%%%%%%%%%%%%%%%%%%%%%%%%%%%%%%%%%%%%%%%%%
%%%%%%%%%%%%%%%%%%%%%%%%%%%%%%%%%%%%%%%%%%%%%%%%%%%%%%%%%%

A graph $\Gamma$ is \emph{cubic} if each of its vertices has valence three.  A \emph{Tait coloring} of a cubic graph is a function $\mathcal C$ from the edge set of $\Gamma$ to the set $\{\text{red}, \text{blue}, \text{green}\}$ such that each vertex is incident to one edge of each color.  Bridge trisections of knotted surfaces in $S^4$ were defined by the first and third authors in~\cite{MZB1} and extended to knotted surfaces in arbitrary 4-manifolds in~\cite{MZB2}.  A \emph{bridge trisection} $\mathcal T$ of a knotted surface $\mathcal K \subset S^4$ is a decomposition
\[ (S^4,\mathcal K) = (X_1,\mathcal D_1) \cup (X_2,\mathcal D_2) \cup (X_3, \mathcal D_3),\]
where, for each $i\in\{1,2,3\}$, $\mathcal D_i$ is a collection of trivial disks in the 4-ball $X_i$, and the pairwise intersection $\tau_{ij} = \D_i \cap \D_j$ is a trivial tangle in the 3-ball $B_{ij} = X_i \cap X_j$.
It follows that the triple intersection $\D_1 \cap \D_2 \cap \D_3$ is a collection $\bold x$ of \emph{bridge points} in the \emph{bridge sphere} $\Sigma = X_1 \cap X_2 \cap X_3$, where $\Sigma$ is a 2-sphere.

The union $\Gamma = \tau_{12} \cup \tau_{23} \cup \tau_{31}$ along the points $\bold x$ is a 1-complex (that is, a graph), which we will call the \emph{1-skeleton} associated to $\mathcal T$.  Observe that $\Gamma$ is cubic, and it has a natural Tait coloring $\mathcal C$ obtained by coloring the arcs of $\tau_{12}$ red, the arcs of $\tau_{23}$ blue, and the arcs of $\tau_{31}$ green.  We say that the coloring $\mathcal C$ of $\Gamma$ is \emph{induced} by $\mathcal T$.

In this paper, we prove that the correspondence between bridge trisections and Tait-colored cubic graphs can be reversed:  Given a cubic graph $\Gamma$ with a Tait coloring $\mathcal C$, the subgraph induced by any pair of colors is a collection of disjoint cycles, and attaching 2-cells along every cycle for each of the three pairings gives rise to a surface $S$, which we call the surface \emph{induced by} $\mathcal{C}$.

\begin{theorem}
\label{main}
	If $\Gamma$ is a cubic graph with a Tait coloring $\mathcal C$, then there exists a bridge trisection $\mathcal T$ of an unknotted surface $\mathcal{U} \subset S^4$ such that the 1-skeleton $\mathcal T$ is graph isomorphic to $\Gamma$, with the coloring $\mathcal C$ induced by $\mathcal T$.  Moreover, if the induced surface $S$ is nonorientable, we may choose the embedding of $\mathcal{U}$ to have any possible normal Euler number.
\end{theorem}

To prove the main theorem, we first prove that every Tait-colored cubic graph $\Gamma$ admits a finite sequence of simplifications called \emph{compressions} yielding the theta graph. The theta graph is the 1-skeleton of the 1-bridge trisection of the unknotted 2-sphere.  Next, we show that there is a sequence of modifications to the 1-bridge trisection of the unknotted 2-sphere that undoes the sequence of compressions, eventually resulting in a bridge trisection of an unknotted surface inducing $\Gamma$.  One of the moves, elementary perturbation, is well-known (see~\cite{MZB1}), while the other two, crosscap summation and tubing, are new constructions that may be of independent interest.

To get a sense of the difficulty of the problem, the motivated reader is encouraged to attempt their own ad hoc construction of a bridge trisection with 1-skeleton isomorphic to either of the two examples shown in Figure~\ref{examples}.  The graph in Figure~\ref{ex1} is the famous Heawood graph~\cite{heawood}; it induces an orientable surface, so we call it $\Gamma_O$.  The graph in Figure~\ref{ex2} appears to be unnamed; it induces a nonorientable surface, so we call it $\Gamma_N$.  We will refer back to these examples throughout the paper.  In particular, the proof of the main theorem is constructive and is carried out for these two examples in Section~\ref{examp}.

\begin{figure}[h!]
\begin{subfigure}{.5\textwidth}
  \centering
  \includegraphics[width=.5\linewidth]{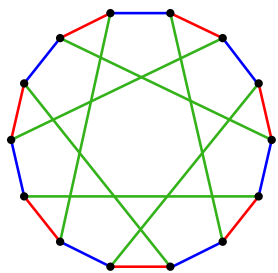}
  \caption{The Heawood graph $\Gamma_O$}
  \label{ex1}
\end{subfigure}%
\begin{subfigure}{.5\textwidth}
  \centering
  \includegraphics[width=.5\linewidth]{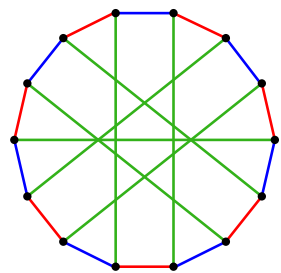}
  \caption{The nonorientable graph $\Gamma_N$}
  \label{ex2}
\end{subfigure}
\caption{Two examples of Tait-colored cubic graphs}
\label{examples}
\end{figure}

A bridge trisection $\mathcal T$ induces additional structure:  By choosing disks $E_{ij} \subset H_{ij}$ with a common boundary curve containing the bridge points $\bold x$, we can project $\tau_{ij}$ onto $E_{ij}$ to obtain a \emph{tri-plane diagram} $\Pau = (\Pau_{12},\Pau_{23},\Pau_{31})$, a triple of planar diagrams of trivial tangles such that any pairwise union $\Pau_{ij} \cup \overline\Pau_{jk}$ yields a classical diagram for an unlink.  Theorem~1.7 from~\cite{MZB1} asserts that any two tri-plane diagrams corresponding to the same bridge trisection $\mathcal T$ are related by interior Reidemeister moves and mutual braid transpositions.

For tri-plane diagrams, we can leverage Theorem~\ref{main} to obtain the following corollary.

\begin{corollary}\label{changes}
	Every tri-plane diagram $\mathcal P$ of a knotted surface $\mathcal K \subset S^4$ can be converted to a tri-plane diagram $\mathcal P'$ for an unknotted surface $\mathcal U$ by a sequence of interior Reidemeister moves and crossing changes.
\end{corollary}

It is natural to wonder if Corollary~\ref{changes} remains true when Reidemeister moves are disallowed when converting the given diagram $\Pau$ to the diagram $\Pau'$ corresponding to the unknotted surface; see Remark~\ref{rmk:crossings} following the proof of Corollary~\ref{changes}.

\begin{question}
	Does every tri-plane diagram $\Pau$ admit a sequence of crossing changes converting it to a tri-plane diagram $\Pau'$ for an unknotted surface?
\end{question}

Finally, Kronheimer and Mrowka~\cite{KM} have recently outlined a plan to give a new topological proof of the famous Four Color Theorem~\cite{AH}.  We offer an alternate route, recalling that Tait's reformulation of the Four Color Theorem states that every bridgeless planar cubic graph admits a Tait coloring~\cite{Tait}. (See~\cite{Brown} for a topological proof of this reformulation.)

\begin{corollary}
The Four Color Theorem is equivalent to the assertion that every bridgeless planar cubic graph is isomorphic to the 1-skeleton of a bridge trisection of an unknotted surface.
\end{corollary}

\begin{proof}
	Suppose $\Gamma$ is a bridgeless planar cubic graph.  If $\Gamma$ is the 1-skeleton of a bridge trisection, then $\Gamma$ inherits a Tait coloring.  Conversely, suppose $\Gamma$ has a Tait coloring.  Then $\Gamma$ is the 1-skeleton of a bridge trisection of an unknotted surface by Theorem~\ref{main}.
\end{proof}

%Thus, it is possible that an expansion of the techniques of this paper could lead to another approach to the Four Color Theorem.

We proceed as follows:  In Section~\ref{prelim}, we set up some background material related to bridge trisections, knotted surfaces, and cubic Tait-colored graphs.  In Section~\ref{ops}, we discuss connected summation, elementary perturbation, crosscap summation, and tubing of bridge trisections, and we define compression of Tait-colored cubic graphs.  In Section~\ref{mainproof}, we prove Theorem~\ref{main} and Corollary~\ref{changes}.  Finally, in Section~\ref{examp}, we carry out the process described in the proof of Theorem~\ref{main} for the examples in Figure~\ref{examples}.

\subsection*{Acknowledgements}

The authors are grateful to the Banff International Research Station for hosting the workshop ``Unifying 4-dimensional knot theory," during which part of this work was completed.  JM is supported by NSF grant DMS-1933019, AT is supported by NSF grant DMS-1664587, and AZ is supported by NSF grants DMS-164578 and DMS-2005518. 

%%%%%%%%%%%%%%%%%%%%%%%%%%%%%%%%%%%%%%%%%%%%%%%%%%%%%%%%%%
%%%%%%%%%%%%%%%%%%%%%%%%%%%%%%%%%%%%%%%%%%%%%%%%%%%%%%%%%%
\section{Preliminaries}\label{prelim}
%%%%%%%%%%%%%%%%%%%%%%%%%%%%%%%%%%%%%%%%%%%%%%%%%%%%%%%%%%
%%%%%%%%%%%%%%%%%%%%%%%%%%%%%%%%%%%%%%%%%%%%%%%%%%%%%%%%%%

%%%%%%%%%%%%%%%%%%%%%%%%%%%%%%%%%%%%%%%%%%%%%%%%%%%%%%%%%%
\subsection{Bridge trisections}
%%%%%%%%%%%%%%%%%%%%%%%%%%%%%%%%%%%%%%%%%%%%%%%%%%%%%%%%%%

%A feature of trisection theory (as introduced in~\cite{GK}) is that every trisection $\mathcal T$ can be encoded by a triple of cut systems $\A$, $\n$, $\g$ in a surface $\Sigma$ with the property that pairs of cut systems yield Heegaard diagrams for the 3-manifolds $\partial X_i\cong\#^{k_i}(S^1 \X S^2)$.  Conversely, such a triple determines the 3-dimensional handlebody components of a trisection (their union is called a \emph{spine}), and the 4-dimensional 1-handlebodies can be attached uniquely by Laudenbach-Poenaru~\cite{LP}.

%An analogous structure exists for bridge trisections.
In~\cite{MZB1}, the first and third author proved that every knotted surface $(X,\Kk)$ admits a bridge trisection $\T$, which can be encoded by a \emph{shadow diagram}, a triple $(a,b,c)$ such that each of $a$, $b$, and $c$ is an embedded collection of pairwise disjoint arcs in $\Sigma$ resulting from pushing the trivial arcs in each pairwise intersection $(B_{ij},\tau_{ij})$ into $\Sigma$. For two simple examples of shadow diagrams, see Figure~\ref{rpfigs}.
 Any two shadow diagrams for the same bridge trisection are related by a sequence of shadow slides (replacing an arc with its band sum with boundary of a neighborhood of another shadow in the same set).
See Figure~\ref{nxb} for some examples of shadow slides.

\begin{figure}[h!]
\begin{subfigure}{.5\textwidth}
  \centering
  \includegraphics[width=.5\linewidth]{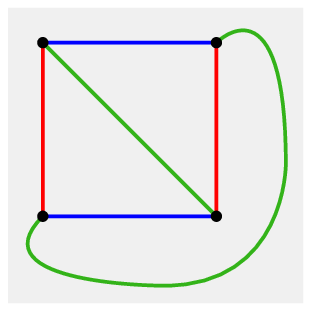}
  \caption{$\F^+$}
  \label{rpp}
\end{subfigure}%
\begin{subfigure}{.5\textwidth}
  \centering
  \includegraphics[width=.5\linewidth]{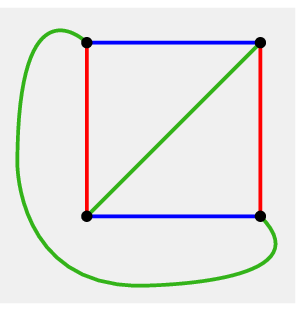}
  \caption{$\F^-$}
  \label{rpm}
\end{subfigure}
\caption{Examples of genus zero shadow diagrams.}
\label{rpfigs}
\end{figure}

%%%%%%%%%%%%%%%%%%%%%%%%%%%%%%%%%%%%%%%%%%%%%%%%%%%%%%%%%%
\subsection{Unknotted surfaces}\label{unkn}
%%%%%%%%%%%%%%%%%%%%%%%%%%%%%%%%%%%%%%%%%%%%%%%%%%%%%%%%%%

We say that an orientable surface $\K \subset S^4$ is \emph{unknotted} if $\K$ is the boundary of a smoothly embedded 3-dimensional handlebody.  Equivalently, $\K$ is unknotted if and only if $\K$ is isotopic into $S^3\subset S^4$~\cite[Theorem 1.2]{KH}. (This is analogous to the fact that a classical knot $K \subset S^3$ is the unknot if and only if $K$ is isotopic into $S^2\subset S^3$.)  For a nonorientable surface $\K$, the situation is slightly more complicated, but $\K$ is unknotted if $\K$ is \emph{almost} isotopic into $S^3\subset S^4$.  Each shadow diagram in Figure~\ref{rpfigs} corresponds to a bridge trisection of a embedding of $\RP^2$ into $S^4$.  We call these two embeddings $\F^+$ and $\F^-$ as shown; they are the two unknotted embeddings of $\RP^2$, where the normal Euler number $e(\F^{\pm})$ satisfies $e(\F^{\pm}) = \pm 2$.  Following~\cite{KH}, we say that a nonorientable surface $\K \subset S^4$ is \emph{unknotted} if $\K$ is isotopic to a connected sum of copies of $\F^+$ and $\F^-$.  Since normal Euler number is additive under connected sum, for unknotted surfaces with nonorientable genus $g$, we have $e(\K) \in \{-2g, -2g + 4,\dots 2g -4 , 2g\}$.  The Whitney-Massey Theorem asserts that the normal Euler number of any embedded surface of nonorientable genus $g$ also falls into this range~\cite{Massey}.

%%%%%%%%%%%%%%%%%%%%%%%%%%%%%%%%%%%%%%%%%%%%%%%%%%%%%%%%%%
\subsection{Cubic graphs and surfaces}
%%%%%%%%%%%%%%%%%%%%%%%%%%%%%%%%%%%%%%%%%%%%%%%%%%%%%%%%%%

All cubic graphs in this paper are assumed to be connected.  We allow our cubic graphs to have parallel edges.  If $\Gamma$ is not the theta graph -- i.e., the graph with two vertices and three parallel edges -- then each edge is parallel to at most one other edge since $\Gamma$ is cubic and connected.

Given a cubic graph $\Gamma$ with a Tait coloring $\mathcal{C}$, recall that $\Gamma$ and $\mathcal{C}$ give rise to the \emph{induced surface} $S$ obtained by attaching 2-cells to $\Gamma$ along each bi-colored cycle determined by $\mathcal{C}$.  See Figure~\ref{example1ab} below for an example.

\begin{figure}[h!]
\begin{subfigure}{.5\textwidth}
  \centering
  \includegraphics[width=.6\linewidth]{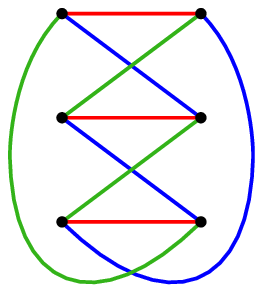}
  \caption{A Tait coloring $\mathcal{C}$ of $K_{3,3}$}
  \label{example1a}
\end{subfigure}%
\begin{subfigure}{.5\textwidth}
  \centering
  \includegraphics[width=.6\linewidth]{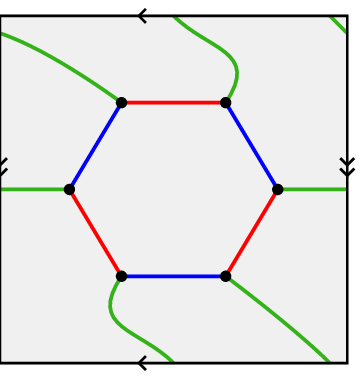}
  \caption{The surface induced by $K_{3,3}$ and $\mathcal{C}$}
  \label{example1b}
\end{subfigure}
\caption{An example of an induced surface.}
\label{example1ab}
\end{figure}

The \emph{patch numbers} $(p_1,p_2,p_3)$ of a Tait-colored cubic graph $\Gamma$ count the number of each type of bi-colored cycle.  If $p_1 = p_2 = p_3=p$, then we simply say that $\T$ is \emph{$p$-patch}.  Both examples $\Gamma_O$ and $\Gamma_N$ shown in Figure~\ref{examples} are 1-patch.  We also keep track of the orientability of the induced surface $S$, which can be verified by an easy condition, offered by the next lemma.

\begin{lemma}\label{orient}
Given a Tait-colored cubic graph $\Gamma$, the induced surface $S$ is orientable if and only if $\Gamma$ is bipartite.
\end{lemma}

\begin{proof}
First, suppose that $S$ is oriented.  The orientation of $S$ induces an orientation of each 2-cell bounded by the bi-colored cycles of $\Gamma$ with respect to $\mathcal{C}$.  Orient the red edges to agree with the orientation of the 2-cells bounded by red-blue cycles, orient the blue edges to agree with the orientation of the blue-green 2-cells, and orient the green edges to agree with the orientation of the green-red 2-cells.  Then the orientations of the blue edges, green edges, and red edges \emph{disagree} with the orientations of the red-blue 2-cells, blue-green 2-cells, and green-red 2-cells, respectively.  It follows that every vertex $v$ in $\Gamma$ is either at the head of a red, blue, and green edge, or $v$ is at the tail of a red, blue, and green edge.  Letting $V^+$ denote the vertices at the heads of a triple of edges and $V^-$ the vertices at the tails, we have that $\Gamma$ is bipartite.

Conversely, suppose that $\Gamma$ is bipartite, with vertices partitioned into $V^+$ and $V^-$.  Orient the edges of $\Gamma$ so that each has a vertex in $V^-$ at its tail and a vertex in $V^+$ at its head.  Finally, orient the red-blue 2-cells so that they agree with the orientations of the red edges, orient the blue-green 2-cells to agree with the orientations of the blue edges, and orient the green-red 2-cells to agree with the orientations of the green edges.  Then (as above) the orientations of blue edges, green edges, and red edges disagree with the orientations of the red-blue 2-cells, the blue-green 2-cells, and the green-red 2-cells, respectively.  This implies that 2-cells are glued along oppositely oriented edges, so that the orientations of the 2-cells agree wherever they overlap.  We conclude that their union, the surface $S$, is orientable.
\end{proof}

Following Lemma~\ref{orient}, we will say a Tait-colored cubic graph $\Gamma$ is \emph{orientable} if $\Gamma$ is bipartite and \emph{nonorientable} otherwise. 

\begin{remark}
Given a Tait-colored cubic graph $\Gamma$, the induced surface $S$ has a cell decomposition with $\Gamma$ as its 1-skeleton and with $p_1 + p_2 + p_3$ 2-cells; hence $\chi(S) = \chi(\Gamma) + p_1 + p_2 + p_3$.  The graph $\Gamma_O$ in Figure~\ref{ex1} is orientable and induces a surface $S_O$, whereas the graph $\Gamma_N$ in Figure~\ref{ex2} is nonorientable and induces $S_N$, satisfying $\chi(S_O) = \chi(S_N) = -4$.
\end{remark}

\begin{remark}
 A somewhat surprising consequence of Lemma~\ref{orient} is that while the Euler characteristic of the induced surface $S$ depends on a choice of Tait coloring, the orientability of $S$ depends only on the underlying graph.  In Figure~\ref{color}, we depict two different Tait colorings $\mathcal C$ and $\mathcal C'$ of a graph $\Gamma$, inducing surfaces $S$ and $S'$, respectively, with $\chi(S) = 2$ and $\chi(S') = 0$.  
\end{remark}

\begin{figure}[h!]
\begin{subfigure}{.5\textwidth}
  \centering
  \includegraphics[width=.5\linewidth]{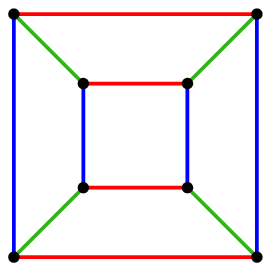}
  \caption{Tait coloring $\mathcal{C}$ of $\Gamma$ inducing a 2-sphere $S$}
  \label{color1}
\end{subfigure}%
\begin{subfigure}{.5\textwidth}
  \centering
  \includegraphics[width=.5\linewidth]{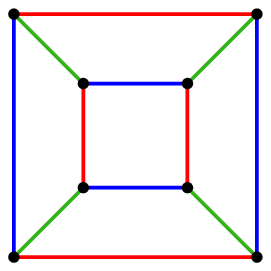}
  \caption{Tait coloring $\mathcal{C}'$ of $\Gamma$ inducing a torus $S'$}
  \label{color2}
\end{subfigure}
\caption{Distinct Tait colorings of $\Gamma$ inducing different surfaces}
\label{color}
\end{figure}

\begin{remark}
If $\Gamma$ is a cubic graph embedded in a surface $S$ such that $S \setminus \Gamma$ is a collection of disks, it does \emph{not} necessarily imply that $\Gamma$ has a Tait coloring inducing the surface $S$.  Indeed, consider the embedding of $K_{3,3}$ in $\RP^2$ shown in Figure~\ref{K3RP}, where $\RP^2$ is obtained from the disk by identifying antipodal points on its boundary.  Since $\RP^2$ is nonorientable and $K_{3,3}$ is bipartite, Lemma~\ref{orient} implies that there does not exist a Tait coloring $\mathcal{C}$ of $K_{3,3}$ inducing $\RP^2$.  We leave the following as an exercise for the reader:  Let $\Gamma \subset S$ be an embedded cubic graph cutting $S$ into disks.  Then $\Gamma$ admits a Tait coloring inducing $S$ if and only if its graph dual $\Gamma'$ admits a 3-coloring (of its vertex set).  In the example shown in Figure~\ref{K3RP}, the dual $\Gamma'$ contains a $K_4$ subgraph.
\end{remark}

\begin{figure}[h!]
  \centering
  \includegraphics[width=.25\linewidth]{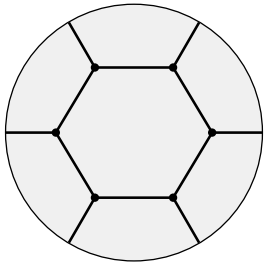}
  \caption{Embedding of $K_{3,3}$ in $\RP^2$, where antipodal points of the disk are identified.  No Tait coloring of $K_{3,3}$ induces $\RP^2$.}
\label{K3RP}
\end{figure}

\section{Operations on bridge trisections and cubic graphs}\label{ops}

In this section, we describe two well-known operations (connected summation and elementary perturbation) and two novel operations (crosscap summation and tubing) that increase the number of bridge points in a given bridge trisection.  We also discuss a way to simplify a cubic graph, called compression.  These operations will be the basis for the construction in the proof of Theorem~\ref{main}.

\subsection{Connected summation}

Given two bridge trisections $\T_1$ of $(X_1,\K_1)$ and $\T_2$ of $(X_2,\K_2)$, with bridge spheres $\Sigma_1$ and $\Sigma_2$ and distinguished bridge points $x_1 \in \bold x_1$ and $x_2 \in \bold x_2$, the connected sum $\T_1 \# \T_2$ is the trisection for $(X_1 \# X_2, \K_1 \# K_2)$ obtained by removing a 4-ball neighborhood of each point $x_i$, which necessarily meets each component piece of $\T_i$ in a ball of the appropriate dimension, then identifying the component pieces of $\T_1$ with $\T_2$ along the resulting boundaries.  On a diagrammatic level, a shadow diagram for $\T_1 \# \T_2$ is obtained by removing disk neighborhoods of the bridge points and gluing the two diagrams along the resulting boundaries.  If $\Gamma_1$ and $\Gamma_2$ are the 1-skeleta of $\T_1$ and $\T_2$, then the 1-skeleton of $\T_1 \# \T_2$ is obtained by vertex summing $\Gamma_1$ and $\Gamma_2$ along the vertices $v_1$ and $v_2$ corresponding to the bridge points $x_1$ and $x_2$.  An example is shown in Figure~\ref{lem3}.

\subsection{Elementary perturbation}

Another previously-known operation on bridge trisections is \emph{elementary perturbation}.  Every genus zero bridge trisection admits a shadow diagram $(a,b,c)$ in which any one of the pairings, say $(a,b)$, can be assumed to be \emph{standard}, meaning that the union $a \cup b$ is a collection of embedded, polygonal curves that bound pairwise disjoint disks in $\Sigma$~\cite{MZB1}.  Choose a disk $D$ bounded by one of the trivial curves in $a\cup b$, and let $\delta$ be an arc in $D$ with one endpoint in the interior of an arc $a_i \in a$ and the other endpoint in the interior of an arc $b_j \in b$.

Now, replace $a_i \cup b_j \cup \delta$ via an IH-move, which changes $a_i$ to two arcs $a_i'$ and $a_i''$, converts $b_j$ to two arcs $b_j'$ and $b_j''$, and replaces $\delta$ with transverse arc $d'$.  Let $a' = a \setminus \{a_i\} \cup \{a_i',a_i''\}$, $b' = b \setminus \{b_j\} \cup \{b_j',b_j''\}$, and $c' = c \cup \{d'\}$.  Then $(a',b',c')$ is a shadow diagram for another bridge trisection $\T'$ of the same surface.  This is called an \emph{elementary perturbation} of $\T$.  See Figure~\ref{pert}.  Of course, our choice of pairing was arbitrary, and this construction will work for any of the three pairings.

\begin{figure}[h!]
\begin{subfigure}{.5\textwidth}
  \centering
  \includegraphics[width=.5\linewidth]{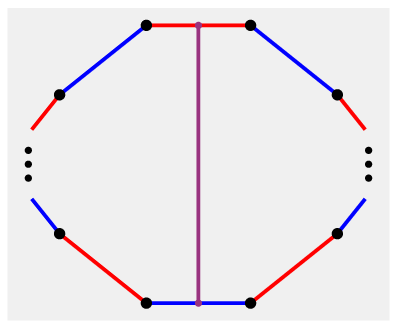}
  \caption{Before}
  \label{per1}
\end{subfigure}%
\begin{subfigure}{.5\textwidth}
  \centering
  \includegraphics[width=.5\linewidth]{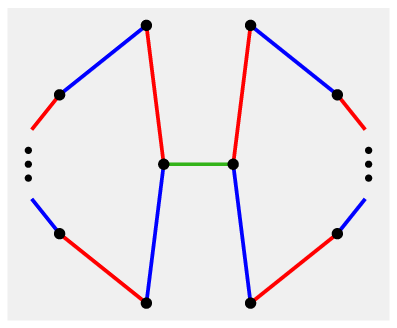}
  \caption{After}
  \label{per2}
\end{subfigure}
\caption{Local pictures of a shadow diagram corresponding to an elementary perturbation.}
\label{pert}
\end{figure}

The reverse operation is called \emph{elementary deperturbation}:  Suppose that a trisection $\T$ has a shadow diagram $(a,b,c)$ such that the pairing $(a,b)$ is standard, and there is an arc $d$ in $c$ and disjoint disks $D_1$ and $D_2$ bounded by arcs in $a \cup b$ such that $d$ meets $D_i$ in a single bridge point $x_i$.  Then $x_i$ is the endpoint of arcs $a_i$ and $b_i$.  Let $a_1'$ be the arc $a_1 \cup d \cup a_2$ and let $b_1'$ be the arc $b_1 \cup d \cup b_2$.  Then, with $a' = a \setminus \{a_1,a_2\} \cup \{a_1'\}$, $b' = b \setminus \{b_1,b_2\} \cup \{b_1'\}$, and $c' = c \setminus \{d\}$, we have that $(a',b',c')$ is a shadow diagram for another bridge trisection $\T'$ of the same surface and 4-manifold.  We say that $\T'$ is related to $\T$ by \emph{elementary deperturbation}.  By inspection, the net result of an elementary perturbation followed by an elementary deperturbation along the appropriate arc returns the original bridge trisection.  For further details on these operations, see~\cite{MZB1} and~\cite{MZB2}.

%On the abstract level, the graph in Figure~\ref{per2} is said to be obtained from the graph in Figure~\ref{per1} by \emph{abstract elementary perturbation}.  By definition, if $\T'$ is obtained from $\T$ by an elementary perturbation, then the 1-skeleton $\Gamma'$ is obtained from $\Gamma$ by an abstract elementary perturbation.  In the reverse direction, suppose that $\Gamma$ is a Tait-colored cubic graph with a (say, green) edge $e_g$ connected two distinct red-blue cycles.  Suppose further that $e_g$ shares a vertex $v'$ with red edge $e_r'$ and blue edge $e_b'$ and another vertex $v''$ with red edge $e_r''$ and blue edge $e_b''$.  We obtain a new Tait-colored cubic graph $\Gamma'$ by removing the vertices $v'$ and $v''$ and the edge $e_g$, replacing $e_r'$ and $e_r''$ with a single edge $e_r$, and replacing $e_b'$ and $e_b''$ with a single edge $e_b$.  This reversal can be seen going right to left in Figure~\ref{pert}, and we say that $\Gamma'$ is obtained from $\Gamma$ by an \emph{abstract elementary deperturbation}.  As above, this move corresponds precisely to the change at the level of 1-skeleta induced by an elementary deperturbation of a bridge trisection.  A critical observation is that abstract perturbation and deperturbation are also inverses at the graph level.

\subsection{Crosscap summation}

Here we introduce a new type of local modification of a bridge trisection, which we call crosscap summation.  As in the definition of elementary perturbation, suppose that a bridge trisection $\T$ admits a shadow diagram $(a,b,c)$ in which $(a,b)$ is standard.  Choose a disk $D$ bounded by one of the trivial curves in $a\cup b$, and let $\delta$ be an arc in $D$ with one endpoint in the interior of an arc $a_i \in a$ and the other endpoint in the interior of an arc $b_j \in b$.

Next, replace $a_i \cup b_j \cup \delta$ via the following procedure:  Introduce two new bridge points $x_+$ and $x_-$ along $\delta$, with a subarc $d'$ of $\delta$ connecting them.  Suppose the endpoints of $a_i$ are $x_a'$ and $x_a''$ and the endpoints of $b_j$ are $x'_b$ and $x''_b$, where $x'_a$ and $x'_b$ can be connected by an arc that does not separate $x_a''$ from $x_b''$ in $D$.  Let $a_i'$ be an arc connecting $x'_a$ to $x_+$, let $a_i''$ be an arc connecting $x_a''$ to $x_-$, let $b_j'$ be an arc connecting $x_b'$ to $x_-$, and let $b_j''$ be an arc connecting $x_b''$ to $x_+$, where either $a_i' \cap b_j' = \emp$ and $a_i''$ meets $b_j''$ in a single point, or vice versa.  Let $a' = a \setminus \{a_i\} \cup \{a_i',a_i''\}$, $b' = b \setminus \{b_j\} \cup \{b_j',b_j''\}$, and $c' = c \cup \{d'\}$.  We say that the resulting triple $(a',b',c')$ is obtained from $(a,b,c)$ by \emph{crosscap summation}.  The choices made in the construction give rise to two distinct versions crosscap summations; we recognize this distinction by calling one a \emph{positive crosscap summation} and the other a \emph{negative crosscap summation}.  A depiction of each is shown in Figure~\ref{ccap}.

\begin{figure}[h!]
\begin{subfigure}{.33\textwidth}
  \centering
  \includegraphics[width=.75\linewidth]{perturb1.eps}
  \caption{Before}
  \label{cc1}
\end{subfigure}%
\begin{subfigure}{.33\textwidth}
  \centering
  \includegraphics[width=.75\linewidth]{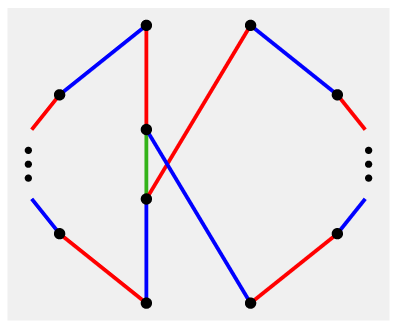}
  \caption{Positive}
  \label{cc2}
\end{subfigure}
\begin{subfigure}{.33\textwidth}
  \centering
  \includegraphics[width=.75\linewidth]{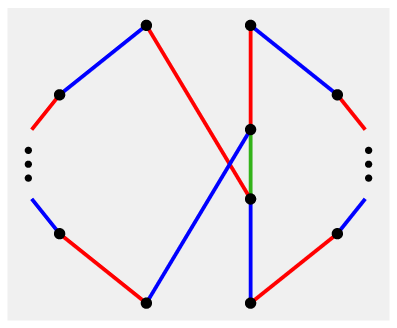}
  \caption{Negative}
  \label{cc3}
\end{subfigure}
\caption{Local pictures of a shadow diagram corresponding to positive and negative crosscap summations.}
\label{ccap}
\end{figure}

We have named this operation crosscap summation because the resulting surface is obtained by taking the connected sum with an unknotted projective plane, as demonstrated in the next lemma.

\begin{lemma}\label{crosscapping}
Suppose that $\Tt$ is a bridge trisection of $\K$ with shadow diagram $(a,b,c)$, where the pairing $(a,b)$ is standard.  Then the result $(a',b',c')$ of positive (resp. negative) crosscap summation is a shadow diagram for a bridge trisection $\T'$ of $\K \# \F^+$ (resp. $\K \# \F^-$).
\end{lemma}

\begin{proof}
Suppose the triple $(a',b',c')$ is obtained from $(a,b,c)$ by a positive crosscap summation along an arc $\delta$.  Let the bridge trisection $\T''$ with shadow diagram $(a'',b'',c'')$ be the result of an elementary perturbation diagram along $\delta$, where $d''$ is the newly created arc in $c''$, with endpoints $v_1$ and $v_2$.  Let $\T^+$ denote the bridge trisection of $\F^+$ shown in Figure~\ref{rpfigs}, and consider the bridge trisection $\T'' \# \T^{\pm}$ of $\K \# \F^{\pm}$, where the connected sum is taken along a neighborhood of the bridge point $v_1$.  Taking the connected sum of shadow diagrams for $\T''$ and $\T^{\pm}$, we can see that the $\T'' \# \T^{\pm}$ admits an elementary deperturbation along the arc corresponding to $d''$ in the connected sum.  The result of elementary deperturbation is a bridge trisection $\T'$ whose shadow diagram coincides with $(a',b',c')$.  Since the three operations used in this proof result in bridge trisections, the statement of the lemma follows.  See Figure~\ref{threestep}.
\end{proof}

\begin{figure}[h!]
\begin{subfigure}{.24\textwidth}
  \centering
  \includegraphics[width=.9\linewidth]{perturb1.eps}
  \caption{$\T$}
  \label{lem1}
\end{subfigure}%
\begin{subfigure}{.24\textwidth}
  \centering
  \includegraphics[width=.9\linewidth]{perturb2.eps}
  \caption{$\T''$}
  \label{lem2}
\end{subfigure}
\begin{subfigure}{.24\textwidth}
  \centering
  \includegraphics[width=.9\linewidth]{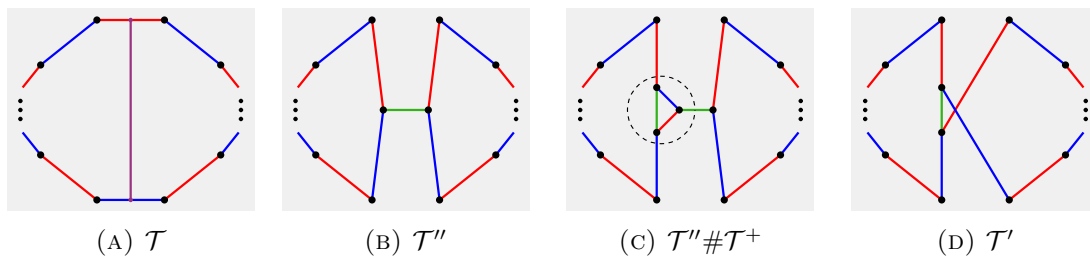}
  \caption{$\T'' \# \T^+$}
  \label{lem3}
\end{subfigure}
\begin{subfigure}{.24\textwidth}
  \centering
  \includegraphics[width=.9\linewidth]{cperturb1.eps}
  \caption{$\T'$}
  \label{lem4}
\end{subfigure}
\caption{The sequence of bridge trisection moves described in the proof of Lemma~\ref{crosscapping}.  Choosing the opposite vertex of $c''$ and the shadow diagram for $\F^-$ in Figure~\ref{rpm} produces the negative crosscap summation shown in Figure~\ref{cc3}.}
\label{threestep}
\end{figure}

\subsection{Tubing}

For this operation, suppose $\T$ is a bridge trisection with shadow diagram $(a,b,c)$ such that one of the pairings, say $(a,b)$ is standard, and let $\delta$ be an arc in $\Sigma$ connecting an arc $a_i$ in one component of $a \cup b$ to an arc $b_j$ in another component, so that $\delta$ meets $a\cup b$ only at its endpoints.  In a move locally identical to a perturbation along $\delta$, we replace $a_i \cup b_j \cup \delta$ via an IH-move, which changes $a_i$ to two arcs $a_i'$ and $a_i''$, converts $b_j$ to two arcs $b_j'$, and $b_j''$, and replaces $\delta$ with transverse arc $d'$.

Let $a' = a \setminus \{a_i\} \cup \{a_i',a_i''\}$, $b' = b \setminus \{b_j\} \cup \{b_j',b_j''\}$, and $c' = c \cup \{d'\}$.  Using the fact that $(a,b,c)$ is a shadow diagram for a bridge trisection $\T$, we can verify by inspection that each of unions $a' \cup b'$, $b' \cup c'$, and $c' \cup a'$ determines a shadow diagram for a bridge splitting of an unlink, and thus $(a',b',c')$ determines a trisection $\T'$ of some knotted surface $\K'$, although the relationship between $\K$ and $\K'$ is not immediately clear.  We will show that $\K'$ is obtained from $\K$ by an operation called 1-handle addition, and to distinguish the operation on knotted surfaces from the operation on bridge trisections, we say that the bridge trisection $\T'$ is related to $\T'$ by \emph{tubing along $\delta$}.  See Figure~\ref{tub} for a local picture of tubing.  Note that the difference between tubing and elementary perturbation is global: in the case of tubing, the arc $\delta$ connects distinct components of $a \cup b$, whereas in an elementary perturbation, $\delta$ connects two arcs in the same component of $a \cup b$.

\begin{figure}[h!]
\begin{subfigure}{.5\textwidth}
  \centering
  \includegraphics[width=.5\linewidth]{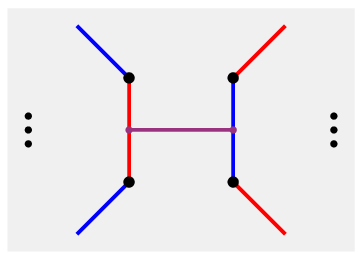}
  \caption{Before}
  \label{tub1}
\end{subfigure}%
\begin{subfigure}{.5\textwidth}
  \centering
  \includegraphics[width=.65\linewidth]{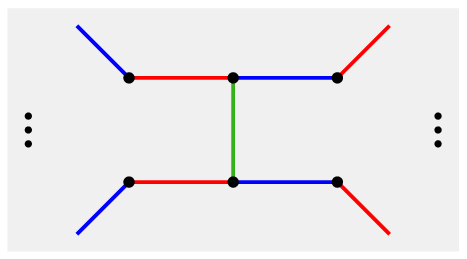}
  \caption{After}
  \label{tub2}
\end{subfigure}
\caption{A local picture of the tubing operation}
\label{tub}
\end{figure}

Suppose now that $\K$ is an embedded surface in $S^4$, and let $D^2 \X I$ be an embedded 1-handle for $\K$, so $(D^2 \X I) \cap \K = D^2 \X \{0,1\}$.  We can obtain a new surface $\K'$ by attaching the 1-handle to $\K$; that is, $\K' = \K \setminus (D^2 \X \{0,1\}) \cup (S^1 \X I)$.  Letting $\delta = \{0\} \X I$, the core of the 1-handle, we say that $\K'$ is obtained from $\K$ by \emph{1-handle addition along $\delta$}~\cite{CKS} (or by \emph{stabilization}~\cite{BS}).  Note that a 1-handle addition can be either orientable or nonorientable.  This handle addition is a well-studied operation in knotted surface theory; it is known, for example, to be an unknotting operation~\cite{HoKa}.  In addition, the following lemma was established for orientable surfaces in~\cite{HoKa} and nonorientable surfaces in~\cite{Kamada} and~\cite{BS}.  We will require this lemma in the next section.

\begin{lemma}\label{unknot}
Suppose $\K \subset S^4$ is unknotted, and $\K'$ is obtained from $\K$ by 1-handle addition.  Then $\K'$ is also unknotted.
\end{lemma}

We now give the connection between tubing and 1-handle addition.

\begin{lemma}\label{tuber}
Suppose $(a,b,c)$ is a shadow diagram for a bridge trisection $\T$ of a knotted surface $\K \subset S^4$, and let $(a',b',c')$ be the shadow diagram for the bridge trisection $\T'$ of the knotted surface $\K'$ obtained by tubing $\T$ along an arc $\delta$.  The $\K'$ is related to $\K$ by a 1-handle addition along $\delta$.
\end{lemma}

\begin{proof}
As in the definition of tubing, suppose that the pairing $(a,b)$ is standard and the arc $\delta \subset \Sigma$ connects arcs $a_i \in a$ and $b_j \in b$.  By an isotopy of $\delta$ we can move it close to two bridge points $x_1$ and $x_2$ as shown in Figure~\ref{tubing1}.  Consider the shadow diagram for a bridge trisection $\T^*$ of a surface $\K^*$ shown in Figure~\ref{tubing2}, which is contained in a neighborhood of $a_i\cup\delta\cup b_j$ in the surface $\Sigma$.  We can see that $\T^*$ is the result of two elementary perturbations applied to the 1-bridge trisection of an unknotted 2-sphere, so that $\K^*$ is an unknotted 2-sphere bounding a 3-ball contained in a neighborhood of $\delta$ in $S^4$, which can be assumed to be disjoint from $\K$.

Now, we perform an operation similar to the connected sum at the bridge points $x_1$ and $x_2$ with the two closest bridge points of $x_1^*$ and $x_2^*$ of $\T^*$, as shown in Figure~\ref{tubing3}.  At the level of the embedded surfaces, this operation corresponds to removing disk neighborhoods of $x_1$ and $x_2$ in $\K$ and disk neighborhoods of $x_1^*$ and $x_2^*$ of $\K^*$ and identifying their respective boundaries to get a new surface $\K'$.  Since $\K^*$ bounds a 3-ball, which is diffeomorphic to $I \X D^2$, we see that $\K'$ is obtained from $\K$ by a 1-handle attachment along $\delta$.  Finally, by inspection, pairs of arc systems in the resulting diagram in Figure~\ref{tubing3} are shadow diagrams for unlinks, yielding a shadow diagram for a bridge trisection $\T'$ of $\K'$, and the diagram in Figure~\ref{tubing3} is isotopic to Figure~\ref{tub2}.
\end{proof}  

\begin{figure}[h!]
\begin{subfigure}{.33\textwidth}
  \centering
  \includegraphics[width=.9\linewidth]{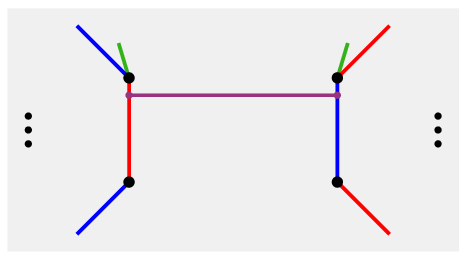}
  \caption{}
  \label{tubing1}
\end{subfigure}%
\begin{subfigure}{.33\textwidth}
  \centering
  \includegraphics[width=.9\linewidth]{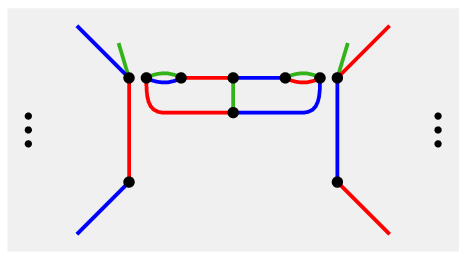}
  \caption{}
  \label{tubing2}
\end{subfigure}
\begin{subfigure}{.33\textwidth}
  \centering
  \includegraphics[width=.9\linewidth]{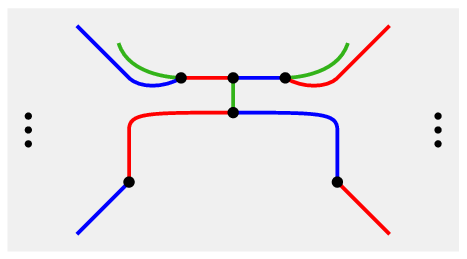}
  \caption{}
  \label{tubing3}
\end{subfigure}
\caption{Shadow diagrams realizing the tubing operation as gluing an unknotted 2-sphere to $\K$ along two disks to get $\K'$, where $\K'$ is obtained from $\K$ by 1-handle addition.}
\label{tubings}
\end{figure}

\subsection{Compression of cubic graphs}

Surprisingly, at the level of the 1-skeleta, the inverses of each of the three moves described above correspond to a single abstract simplification move on a Tait-colored cubic graph $\Gamma$.  Choose an distinguished edge in $\Gamma$, and suppose without loss of generality that the edge, call it $e_g$, is colored green.  We also suppose further that $e_g$ is not parallel to another edge of $\Gamma$.  Let $v^+$ and $v^-$ be the endpoints of $e_g$.  Then there are red edges $e_r^{\pm}$ with one endpoint on $v^{\pm}$ and another endpoint on $v_r^{\pm}$, and there are blue edges $e_b^{\pm}$ with one endpoint on $v^{\pm}$ and another endpoint on $v_b^{\pm}$.  Since $e_g$ is not parallel to any other edge, it follows that the four edges $\{e_r^{\pm},e_b^{\pm})$ are distinct, $v_r^+ \neq v_r^-$, $v_b^+ \neq v_b^-$, and no vertex in the set $\{v_r^{\pm},v_b^{\pm}\}$ is $v^+$ or $v^-$.

We obtain a new Tait-colored cubic graph $\Gamma'$ by removing the vertices $v^{\pm}$ and edge $e_g$, replacing the pair $e_r^+$ and $e_r^-$ with a single red edge $e_r$ between $v_r^+$ to $v_r^-$, and replacing the pair $e_b^+$ and $e_b^-$ with a single blue edge between $v_b^+$ and $v_b^-$.  We say that the new graph $\Gamma'$ is obtained from $\Gamma$ by \emph{compression} along the edge $e_g$.  By inspection, we can see that the graph in Figure~\ref{per1} is obtained from the graph in Figure~\ref{per2} by compression along the displayed green edge.  Similarly, the graphs in Figures~\ref{cc2} and~\ref{cc3} are graph isomorphic, and the graph in Figure~\ref{cc1} is obtained from either of these graphs by compression along the displayed green edge.  Finally, the graph in Figure~\ref{tub1} is obtained from the graph in Figure~\ref{tub2} by compression along the displayed green edge.

These three examples of compression are local identical but globally different, and so we further distinguish them.  For that purpose, we define orientable and nonorientable edges.  Suppose that $\Gamma$ is Tait-colored cubic graph, and $e$ be an edge of $\Gamma$ such that both endpoints of $e$ are contained in the same bi-colored cycle $C$ of the two colors opposite the color of $e$.  Coherently orient $C$, so that the vertices have alternating $+$ and $-$ labels as in the proof of Lemma~\ref{orient}.  If $e$ connects vertices of opposite sign, we say $e$ is \emph{orientation-preserving}.  Otherwise, $e$ connects vertices of the same sign, and we say $e$ is \emph{orientation-reversing}.  Equivalently, $e$ is orientation-preserving if and only if it completes a path in $C$ to a cycle of even length. Following the proof of Lemma~\ref{orient}, we also note that a one-patch graph $\Gamma$ is orientable if and only if it does not contain an orientation-reversing edge with respect to some  bicolored cycle.

By definition, an edge $e$ with vertices in the same bi-colored cycle $C$ of the two opposite colors is either orientation-preserving or orientation-reversing.  If, on the the other hand, $e$ connects distinct bi-colored cycles of the two opposite colors, we say that $e$ is \emph{connecting}.  A compression performed along a connecting edge is called an \emph{p-compression} (Figures~\ref{per1} and~\ref{per2}),  a compression along an orientation-reversing edge is called a \emph{c-compression} (Figures~\ref{cc1}, ~\ref{cc2}, and~\ref{cc3}), and a compression along an orientation-preserving edge is called a \emph{t-compression} (Figures~\ref{tub1} and~\ref{tub2}).  Note that p-compression, c-compression, and t-compression are operations that are inverses to the operations on the 1-skeleton of a bridge trisection induced by elementary perturbation, crosscap summation, and tubing, respectively.

\begin{remark}
\label{rem:ortn}
	We observe that both p-compression and t-compression of an oriented graph result in orientable graphs.  On the other hand, c-compression can only be applied to nonorientable graphs and may result in either an orientable graph or a nonorientable graph.
\end{remark}

\section{Proof of the main theorem}\label{mainproof}

The theta graph is the simplest Tait-colorable graph, and it is the 1-skeleton of the simplest bridge trisection, the 1-bridge trisection of the unknotted $S^2$.  Before proving the main theorem, we require several technical results related to sequences of compressions reducing a given graph.

\begin{proposition}\label{ugh}
Suppose $\Gamma$ is a one-patch nonorientable Tait-colored cubic graph.  Then $\Gamma$ admits a c-compression yielding a nonorientable graph or the theta graph.
\end{proposition}

\begin{proof}
Note that up to isomorphism, $K_4$ has a unique Tait coloring, every edge is nonorientable, and any c-compression yields the theta graph.  We will show that if every c-compression of $\Gamma$ yields an orientable graph, then $\Gamma = K_4$, from which the statement of the proposition follows.  Suppose every c-compression of $\Gamma$ yields an orientable graph.  We produce a convenient picture of $\Gamma$, in which the red and blue edges form a regular $n$-gon, and the green edges are drawn as chords of this $n$-gon (as in the examples in Figure~\ref{examples}).  In this setting, every pair of green edges meets either once or not at all.  Since $\Gamma$ is nonorientable, it contains a nonorientable edge $e$; we suppose without loss of generality that $e$ is colored green.  Orient the vertices of the red-blue cycle $C$ with $+$ and $-$.  Since each orientation-preserving green edge connects two vertices of opposite sign, while each green orientation-reversing edge connects vertices of the same sign, and there are the same number of vertices labeled $+$ as there are labeled $-$, it follows that the number of green orientation-reversing edges is even.  

Removing the red and blue edges incident to $e$ separates the red-blue cycle $C$ of $\Gamma$ into paths $p$ and $p'$.  We assume further that the vertices adjacent to $e$ are labeled $+$, so that both endpoints of the path $p$ and both endpoints of the path $p'$ are labeled $-$.  Consider the graph $\Gamma'$ obtained from c-compression of $\Gamma$ along $e$.  The red-blue cycle $C'$ of $\Gamma'$ is obtained by connecting the paths $p$ and $p'$ along their endpoints, and thus, an orientation of $C'$ can be obtained by preserving the orientation of $p$ coming from $C$ and reversing the orientation of $p'$ coming from $C$.  See Figure~\ref{cpcomp}.

\begin{figure}[h!]
\begin{subfigure}{.5\textwidth}
  \centering
  \includegraphics[width=.5\linewidth]{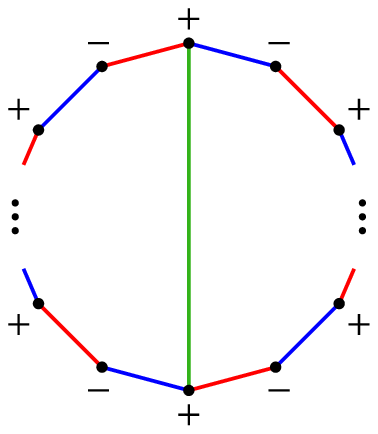}
  %\caption{Before c-compression}
  \label{cpc1}
\end{subfigure}%
\begin{subfigure}{.5\textwidth}
  \centering
  \includegraphics[width=.5\linewidth]{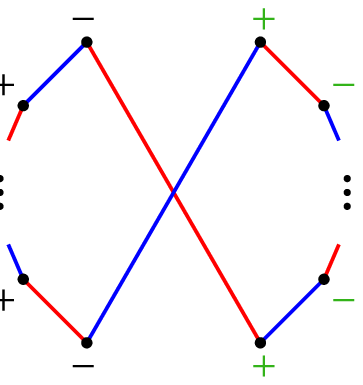}
  %\caption{After c-compression}
  \label{cpc2}
\end{subfigure}
\caption{A nonorientable edge splits a red-blue cycle into paths $p$ and $p'$ (left), where the result of a c-compression preserves the orientation of $p$ and reverses the orientation of $p'$ (right).}
\label{cpcomp}
\end{figure}

It follows that if a green orientation-preserving edge of $\Gamma$ crosses $e$, then that edge becomes orientation-reversing in $\Gamma'$.  Similarly, if a green orientation-reversing edge of $\Gamma$ avoids $e$, then that edge remains orientation-reversing in $\Gamma'$.  By assumption, $\Gamma'$ does not contain an orientation-reversing edge, and thus we see that every green orientation-preserving edge avoids $e$, while every green orientation-reversing edge crosses $e$.  Moreover, this is true not just for $e$ but for every green orientation-reversing edge.  We conclude that every pair of green orientation-reversing edges in $\Gamma$ meet in a single point, and no green orientation-preserving edge crosses a green orientation-reversing edge.

Suppose now that $\Gamma$ contains a green orientation-preserving edge $e'$, and let $\Gamma^*$ be the graph obtained from $\Gamma$ by deleting the green orientation-reversing edges along with any adjacent vertices and incident red or blue edges.  Then $\Gamma^*$ is a proper subgraph of $\Gamma$, and each component of $\Gamma^*$ is spanned by a red-blue path $p^*$, with vertices $w_1,\dots,w_m$ and edges $e_1,\dots,e_{m-1}$ appearing in order.  In $\Gamma$, each $w_i$ is then the endpoint of a green orientation-preserving edge $e^*$, and since $e^*$ crosses no green orientation-reversing edge, the other endpoint of $e^*$ is contained in $\{w_1,\dots,w_m\}$.  It follows that $m$ is even, so that the edges $e_1$ and $e_{m-1}$ are the same color, say red.  Note further that the valence of $w_1$ and $w_m$ in $\Gamma^*$ is two, while the valence of the other vertices is three.  Thus, every vertex of $\Gamma^*$ is the endpoint of both a red and green edge, so $\Gamma^*$ contains a red-green cycle.  Since $\Gamma^*$ is not all of $\Gamma$, we have that $\Gamma$ contains more than one red-green cycle, contradicting the assumption that $\Gamma$ is one-patch.  See Figure~\ref{oedge} for an example.

\begin{figure}[h!]
 \centering
  \includegraphics[width=.3\linewidth]{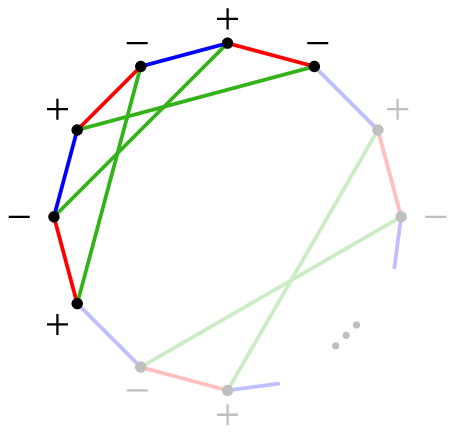}
\caption{If every c-compression of $\Gamma$ along a green edge yields an orientable graph, then a component of the subgraph $\Gamma^*$ of $\Gamma$ induced by the endpoints of green orientation-preserving edges contains a bi-colored cycle.}
\label{oedge}
\end{figure}

We are left with the case that every green edge of $\Gamma$ is orientation-reversing.  As noted above, there are an even number of such edges, so  $\Gamma$ has $4k$ vertices for some integer $k$.  Label the vertices of $\Gamma$ in order, $v_1,\dots,v_{4k}$.  Since every green edge must cross every other green edge, each green edge is a diameter of the red-blue cycle $C$, connecting $v_i$ to $v_{2k+i}$; otherwise, there would exist disjoint green edges.  Consider the green edges connecting $v_1$ to $v_{2k+1}$ and $v_2$ to $v_{2k+2}$.  Since the edges of $C$ alternate colors, it follows that the edge between $v_1$ and $v_2$ is the same color as the edge between $v_{2k+1}$ and $v_{2k+2}$.  Thus, $\Gamma$ contains a bi-colored cycle of length four.  It follows that $\Gamma =K_4$, the unique one-patch graph with four vertices.
\end{proof}

With this technical hurdle out the the way, we can swiftly prove the next lemma.

\begin{lemma}\label{totalcomp}
Suppose $\Gamma$ is a Tait-colored cubic graph
\be
\item If $\Gamma$ is not 1-patch, then $\Gamma$ can be reduced to a 1-patch graph by a finite number of p-compressions.
\item If $\Gamma$ is an orientable 1-patch graph, then $\Gamma$ can be reduced to the theta graph by a finite sequence alternating between t-compressions and p-compressions.
\item If $\Gamma$ is a nonorientable 1-patch graph, then $\Gamma$ can be reduced to the theta graph by a finite number of c-compressions.
\ee
\end{lemma}

\begin{proof}
First, suppose that $\Gamma$ is not 1-patch.  We induct on the sum $p_1 + p_2 + p_3$ of the patch numbers of $\Gamma$.  Suppose without loss of generality that $\Gamma$ contains at least two red-blue cycles.  Since $\Gamma$ is connected, there is some green edge $e_g$ connecting distinct red-blue cycles (and thus $e_g$ is not parallel to another edge).  Then an p-compression of $\Gamma$ along $e_g$ produces a new graph $\Gamma'$ with one fewer red-blue cycle than $\Gamma$.  Since no blue-green nor green-red cycles have been added, the claim holds by induction.

For the second part of the lemma, suppose $\Gamma$ is orientable and 1-patch.  Here we induct on the number of vertices of $\Gamma$.  Suppose that the claim is true for all orientable 1-patch graphs with fewer vertices than $\Gamma$.  By assumption, all edges of $\Gamma$ are orientation-preserving.  Choose one, and let $\Gamma'$ be the result of a t-compression of $\Gamma$ along it.  Note that t-compression decreases the number of vertices of $\Gamma$ by two but increases the patch number by one, so that $\Gamma'$ is not 1-patch, although $\Gamma'$ is spanned by a single bi-colored cycle, so it is still connected.  By the first step, $\Gamma'$ admits an p-compression yielding a new one-patch graph $\Gamma''$ with four fewer vertices than $\Gamma$.  Since t-compression and p-compression of an orientable graph yield another orientable graph (see Remark~\ref{rem:ortn}), the claim holds by induction.

Finally, suppose $\Gamma$ is a nonorientable 1-patch graph.  Again, we induct on the number of vertices of $\Gamma$.  Noting that c-compression of a 1-patch graph yields a 1-patch graph, we have by Proposition~\ref{ugh} that $\Gamma$ has a nonorientable edge such that c-compression along $e$ yields another nonorientable graph, and the third claim follows immediately.
\end{proof}

We remark that following the proof of Proposition~\ref{ugh}, we can take the sequence of c-compressions guaranteed by claim (3) of Lemma~\ref{totalcomp} to occur along edges of the same color.  As an example, seven c-compressions convert the nonorientable one-patch graph $\Gamma_N$ from Figure~\ref{ex1} to the theta graph.  This sequence of compressions is shown in Figure~\ref{nx}.

\begin{figure}[h!]
\begin{subfigure}{.24\textwidth}
  \centering
  \includegraphics[width=.9\linewidth]{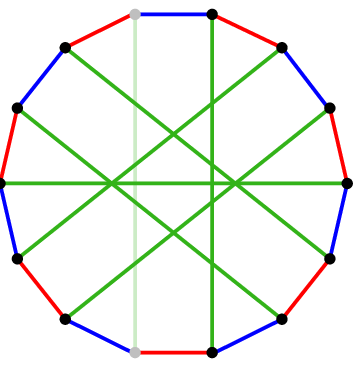}
  \label{nx1}
\end{subfigure}%
\begin{subfigure}{.24\textwidth}
  \centering
  \includegraphics[width=.9\linewidth]{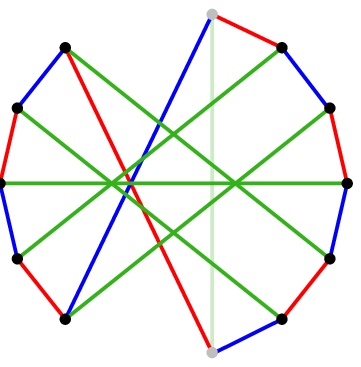}
  \label{nx2}
\end{subfigure}
\begin{subfigure}{.24\textwidth}
  \centering
  \includegraphics[width=.9\linewidth]{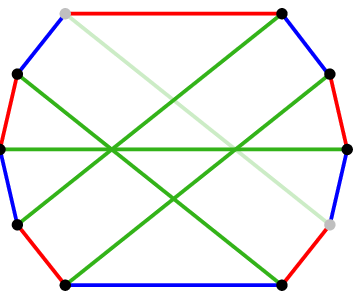}
  \label{nx3}
\end{subfigure}%
\begin{subfigure}{.24\textwidth}
  \centering
  \includegraphics[width=.9\linewidth]{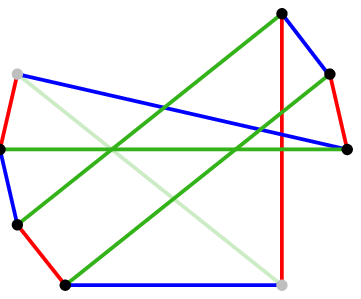}
  \label{nx4}
\end{subfigure}
\begin{subfigure}{.24\textwidth}
  \centering
  \includegraphics[width=.9\linewidth]{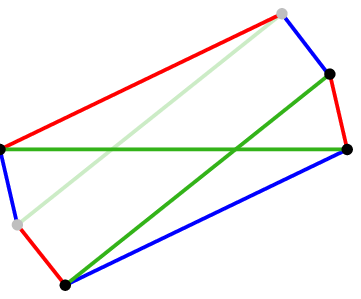}
    \label{nx5}
\end{subfigure}%
\begin{subfigure}{.24\textwidth}
  \centering
  \includegraphics[width=.9\linewidth]{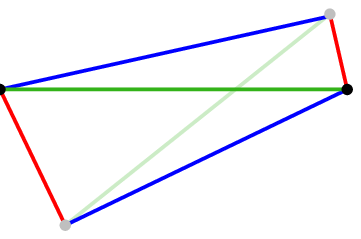}
  \label{nx6}
\end{subfigure}
\begin{subfigure}{.24\textwidth}
  \centering
  \includegraphics[width=.9\linewidth]{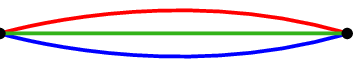}
  \label{nx7}
\end{subfigure}%
\caption{A sequence of c-compressions converting $\Gamma_N$ to the theta graph}
\label{nx}
\end{figure}

We can now prove the main theorem, which we restate here for convenience.

\begin{reptheorem}{main}
	If $\Gamma$ is a cubic graph with a Tait coloring $\mathcal C$, then there exists a bridge trisection $\mathcal T$ of an unknotted surface $\mathcal{U} \subset S^4$ such that the 1-skeleton $\mathcal T$ is graph isomorphic to $\Gamma$, with the coloring $\mathcal C$ induced by $\mathcal T$.  Moreover, if the induced surface $S$ is nonorientable, we may choose the embedding of $\mathcal{U}$ to have any possible normal Euler number.\end{reptheorem}
\begin{proof}
Suppose first that $\Gamma$ is an orientable 1-patch graph.  By Lemma~\ref{totalcomp}, there is a sequence of alternating t-compressions and p-compressions converting $\Gamma$ to the theta graph.  In the reverse direction, there is a sequence of alternating elementary perturbations and tubings performed on the 1-bridge trisection of the unknotted 2-sphere that cancel the t-compressions and p-compressions, so that the 1-skeleton of the resulting bridge trisection $\T$ of a surface $\K \subset S^4$ is isomorphic to $\Gamma$.  By Lemma~\ref{tuber}, the surface $\K$ is obtained from the unknotted 2-sphere by a sequence of 1-handle additions, and thus by repeated applications of Lemma~\ref{unknot}, we have that $\K$ is unknotted.

Next, suppose $\Gamma$ is a nonorientable 1-patch graph.  By Lemma~\ref{totalcomp}, there is a sequence of $n$ c-compressions converting $\Gamma$ to the theta graph.  In the reverse direction, a sequence of $n$ crosscap summations on the 1-bridge trisection of the unknotted 2-sphere cancel the c-compressions, yielding a bridge trisection $\T$ with 1-skeleton isomorphic to $\Gamma$.  By Lemma~\ref{crosscapping}, the knotted surface $\K$ corresponding to $\T$ is the connected sum of $n$ copies of $\F^+$ and $\F^-$, so $\K$ is unknotted.  Moreover, without affecting the induced graph, we can choose the number of each type of summand by picking between positive and negative crosscap summations.  Thus, with these choices we can construct $\K$ to have any possible normal Euler number between $-2n$ and $2n$, which are the only possible values by~\cite{Massey}.

Finally, if $\Gamma$ is not 1-patch, Lemma~\ref{totalcomp} asserts that $\Gamma$ reduces to a 1-patch graph $\Gamma'$ after a sequence of p-compressions.  By the previous steps, there exists a bridge trisection $\T'$ of an unknotted surface $\K$ such that the 1-skeleton of $\T'$ is isomorphic to $\Gamma'$, and if $\Gamma'$ is nonorientable, $\K$ can be chosen with any possible normal Euler number.  Then there is a sequence of elementary perturbations of $\T'$ canceling the sequence of p-compressions, yielding a bridge trisection $\T$ of the same unknotted surface $\K$, where the 1-skeleton of $\T$ is isomorphic to $\Gamma$, completing the proof.
\end{proof} 

Recall that for a bridge trisection $\mathcal T$, a choice of disks $E_{ij} \subset B_{ij}$ containing the bridge points, and a projection of the tangles $\tau_{ij}$ onto $E_{ij}$ with crossing data determines a tri-plane diagram $\Pau = (\Pau_{12},\Pau_{23},\Pau_{31})$ representing $\T$, and any two tri-plane diagrams $\Pau$ and $\Pau'$ for $\T$ are related by interior Reidemeister moves and mutual braid transpositions.  We now prove Corollary~\ref{changes}, which we restate for convenience.

\begin{repcorollary}{changes}
Every tri-plane diagram $\mathcal P$ of an knotted surface $\mathcal K \subset S^4$ can be converted to a tri-plane diagram $\mathcal P'$ for an unknotted surface $\mathcal U$ by a sequence of interior Reidemeister moves and crossing changes.
\end{repcorollary}

\begin{proof}
Suppose that $\T$ is a bridge trisection of an embedded surface $\K$ in $S^4$, where $\Gamma$ is the 1-skeleton of $\T$ with induced Tait coloring $\mathcal{C}$.  By Theorem~\ref{main}, there exists a bridge trisection $\T'$ of an unknotted surface $\K'$ whose Tait-colored 1-skeleton is (graph) isomorphic to $\Gamma$ and $\mathcal{C}$.  Let $\Pau$ and $\Pau'$ be tri-plane diagrams corresponding to $\T$ and $\T'$, respectively.  The graph isomorphism of the 1-skeleta of $\T$ and $\T'$ induces a bijection from the bridge points of $\Pau$ and $\Pau'$.  After performing some number of mutual braid transpositions on the tri-plane diagram $\Pau'$, we may assume that this bijection is the identity.

Viewing the tangles $\tau_{ij}$ and $\tau_{ij'}$ as being contained in the same 3-ball $B$, the graph bijection implies that $\tau_{ij}$ and $\tau_{ij'}$ are homotopic via a homotopy supported outside of a neighborhood of $\pd B$.  Using the projection disk $E_{ij}$, the (generic) projection of this homotopy yields a sequence of interior Reidemeister moves and crossing changes taking $\Pau_{ij}$ to $\Pau_{ij}'$.  Carrying out this process in each of the three sectors yields the corollary.
\end{proof}

\begin{remark}
\label{rmk:crossings}
	Although the end result of the the sequence of interior Reidemeister moves and crossing changes in the proof of Corollary~\ref{changes} yields a tri-plane diagram $\Pau'$, there is no reason to expect that any of the intermediate diagrams is a tri-plane diagram, since changing a single crossing likely destroys the condition that tangles pair to give diagrams of unlinks.
\end{remark}

%\begin{question}
%Can interior Reidemeister moves be removed from the statement of Corollary~\ref{changes}?  In other words, does every tri-plane diagram admit a sequence of only crossing changes converting it to a tri-plane diagram for an unknotted surface?
%\end{question}

\section{The examples $\Sigma_O$ and $\Sigma_N$}\label{examp}

We conclude by working through the details of Theorem~\ref{main} with the two examples from Figure~\ref{examples}.  First, the Tait-colored Heawood graph $\Sigma_O$ is orientable and one-patch, and thus by Lemma~\ref{totalcomp}, it admits an alternating sequence of t-compressions and p-compressions converting it to the theta graph.  These compressions are shown in Figure~\ref{ox}.

\begin{figure}[h!]
\begin{subfigure}{.24\textwidth}
  \centering
  \includegraphics[width=.9\linewidth]{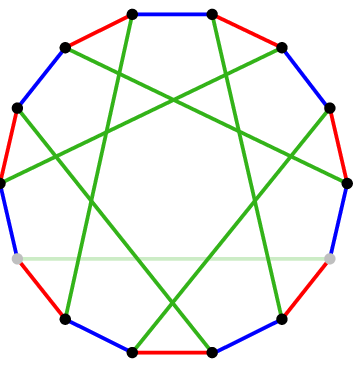}
  \label{ox1}
\end{subfigure}%
\begin{subfigure}{.24\textwidth}
  \centering
  \includegraphics[width=.9\linewidth]{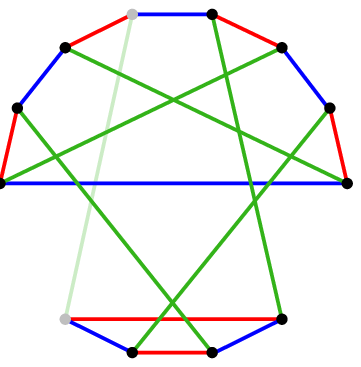}
  \label{ox2}
\end{subfigure}
\begin{subfigure}{.24\textwidth}
  \centering
  \includegraphics[width=.9\linewidth]{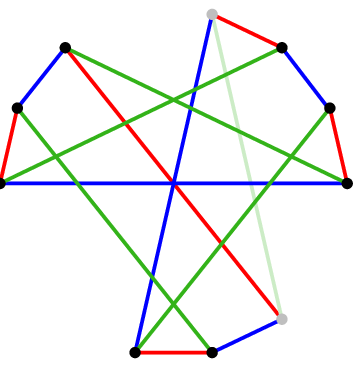}
  \label{ox3}
\end{subfigure}%
\begin{subfigure}{.24\textwidth}
  \centering
  \includegraphics[width=.9\linewidth]{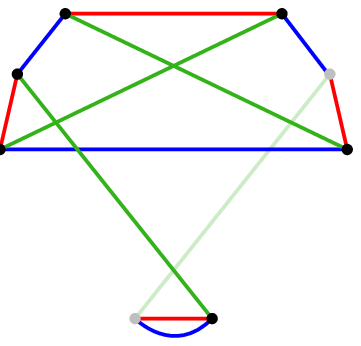}
  \label{ox4}
\end{subfigure}
\begin{subfigure}{.24\textwidth}
  \centering
  \includegraphics[width=.9\linewidth]{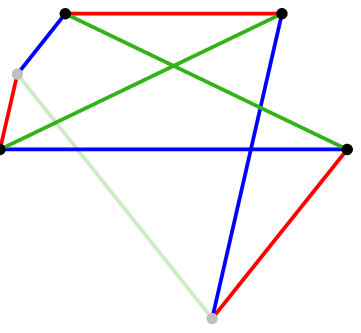}
    \label{ox5}
\end{subfigure}%
\begin{subfigure}{.24\textwidth}
  \centering
  \includegraphics[width=.9\linewidth]{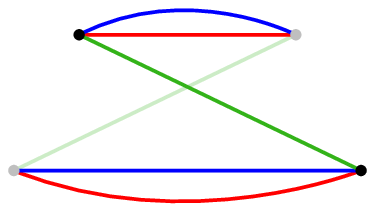}
  \label{ox6}
\end{subfigure}
\begin{subfigure}{.24\textwidth}
  \centering
  \includegraphics[width=.9\linewidth]{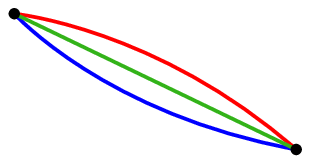}
  \label{ox7}
\end{subfigure}%
\caption{An sequence of t-compressions and p-compressions converting $\Gamma_O$ to the theta graph.}
\label{ox}
\end{figure}

At the level of bridge trisections, we can work our way backwards, starting with the 1-bridge splitting of the unknotted 2-sphere and performing an alternating sequence of elementary perturbations and tubings so that each subfigure of Figure~\ref{oxb} below is graph isomorphic to a corresponding graph in the sequence of compressions shown in Figure~\ref{ox}, in reverse order.

\begin{figure}[h!]
\begin{subfigure}{.24\textwidth}
  \centering
  \includegraphics[width=.9\linewidth]{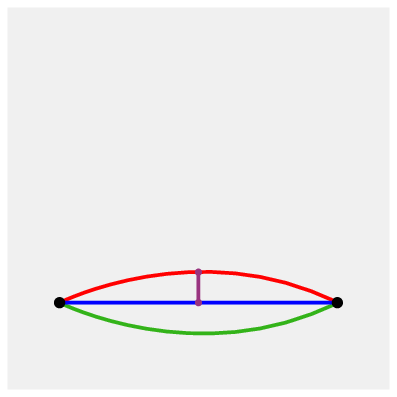}
  \label{oxb1}
\end{subfigure}%
\begin{subfigure}{.24\textwidth}
  \centering
  \includegraphics[width=.9\linewidth]{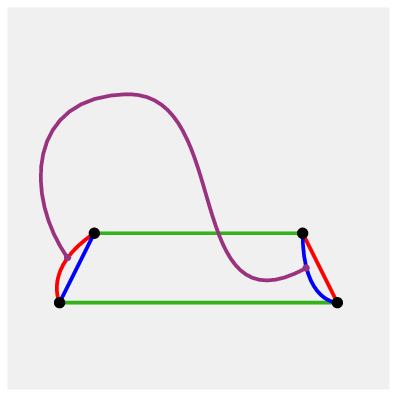}
  \label{oxb2}
\end{subfigure}
\begin{subfigure}{.24\textwidth}
  \centering
  \includegraphics[width=.9\linewidth]{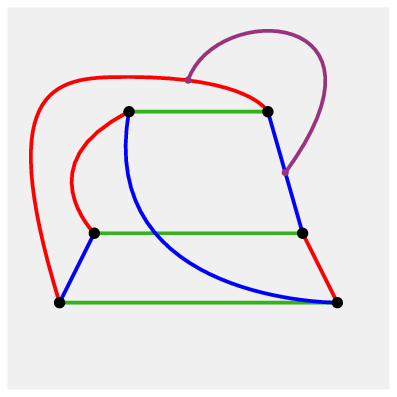}
  \label{oxb3}
\end{subfigure}%
\begin{subfigure}{.24\textwidth}
  \centering
  \includegraphics[width=.9\linewidth]{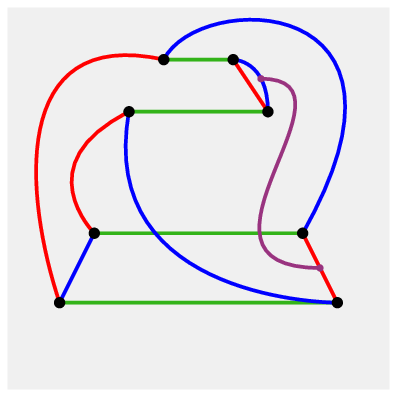}
  \label{oxb4}
\end{subfigure}
\begin{subfigure}{.24\textwidth}
  \centering
  \includegraphics[width=.9\linewidth]{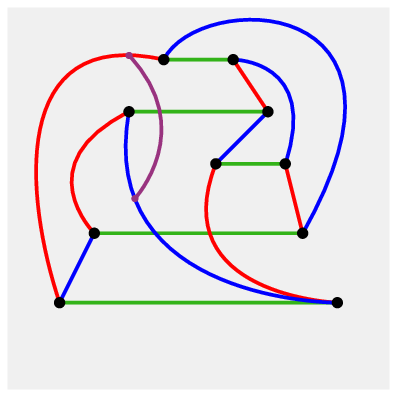}
    \label{oxb5}
\end{subfigure}%
\begin{subfigure}{.24\textwidth}
  \centering
  \includegraphics[width=.9\linewidth]{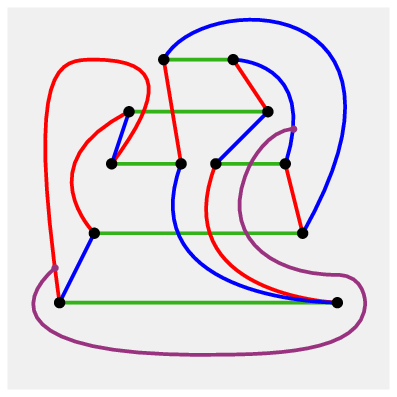}
  \label{oxb6}
\end{subfigure}
\begin{subfigure}{.24\textwidth}
  \centering
  \includegraphics[width=.9\linewidth]{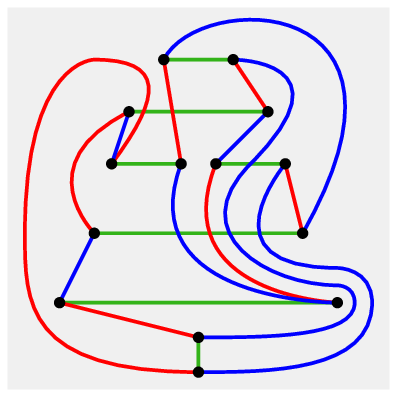}
  \label{oxb7}
\end{subfigure}%
\caption{A sequence of elementary perturbations and tubings performed on the 1-bridge trisection of the unknotted 2-sphere, the end result of which is a bridge trisection of an unknotted genus three surface with 1-skeleton isomorphic to the Heawood graph $\Sigma_O$.}
\label{oxb}
\end{figure}

Turning to the other example, recall that for the nonorientable graph $\Sigma_N$, Theorem~\ref{main} also allows us to choose the normal Euler number of the resulting unknotted surface.  In Figure~\ref{nxb}, we have chosen to reverse the compressions shown in Figure~\ref{nx} by performing three positive crosscap summations followed by three negative crosscap summations, so that the resulting surface $\mathcal{U}$ has normal Euler number zero.  There is an added layer of complexity in this example, since a crosscap summation necessarily introduces shadow arcs that cross each other, and before we perform the next crosscap summation, we are required to first carry out a sequence of shadow slides to convert the red and blue arcs to a standard pairing.

\begin{figure}[h!]
\begin{subfigure}{.24\textwidth}
  \centering
  \includegraphics[width=.9\linewidth]{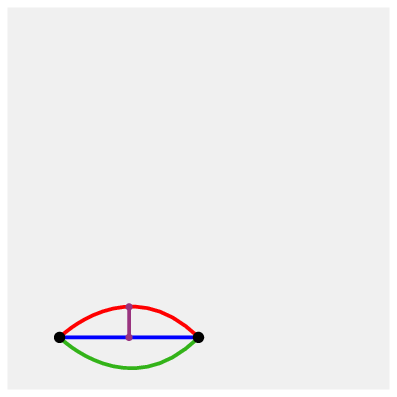}
  \label{nxb1}
\end{subfigure}%
\begin{subfigure}{.24\textwidth}
  \centering
  \includegraphics[width=.9\linewidth]{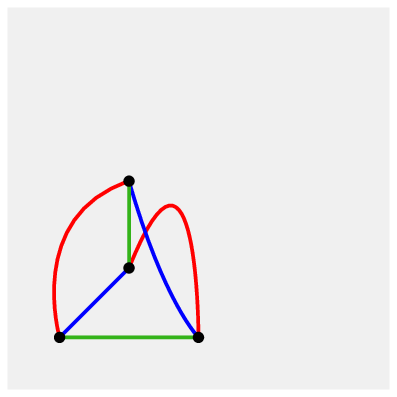}
  \label{nxb2}
\end{subfigure}
\begin{subfigure}{.24\textwidth}
  \centering
  \includegraphics[width=.9\linewidth]{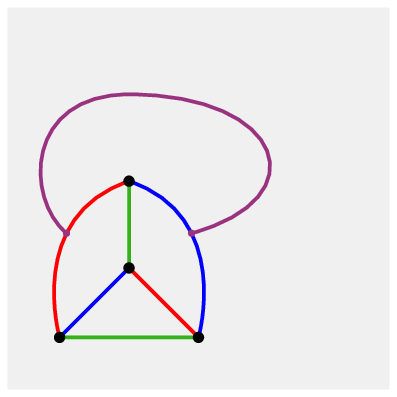}
  \label{nxb3}
\end{subfigure}%
\begin{subfigure}{.24\textwidth}
  \centering
  \includegraphics[width=.9\linewidth]{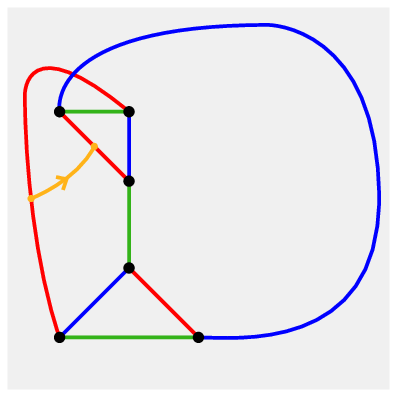}
  \label{nxb4}
\end{subfigure}
\begin{subfigure}{.25\textwidth}
  \centering
  \includegraphics[width=.9\linewidth]{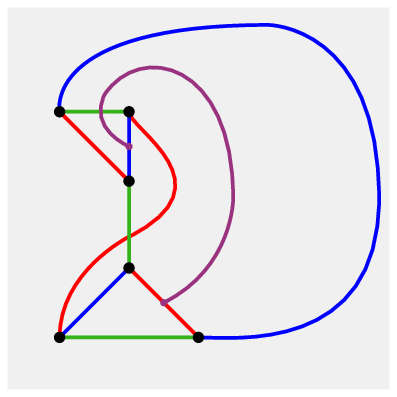}
    \label{nxb5}
\end{subfigure}%
\begin{subfigure}{.25\textwidth}
  \centering
  \includegraphics[width=.9\linewidth]{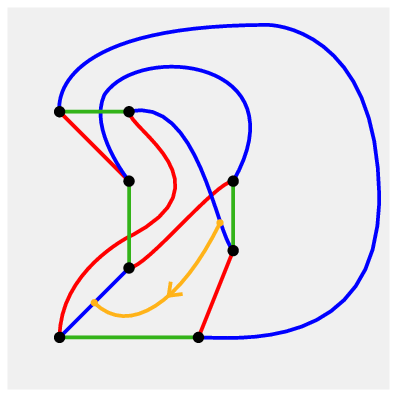}
  \label{nxb6}
\end{subfigure}
\begin{subfigure}{.25\textwidth}
  \centering
  \includegraphics[width=.9\linewidth]{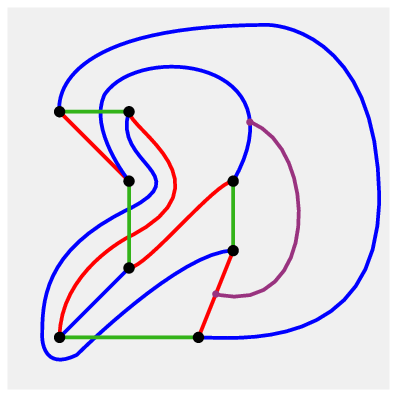}
  \label{nxb7}
  \end{subfigure}
  \begin{subfigure}{.25\textwidth}
  \centering
  \includegraphics[width=.9\linewidth]{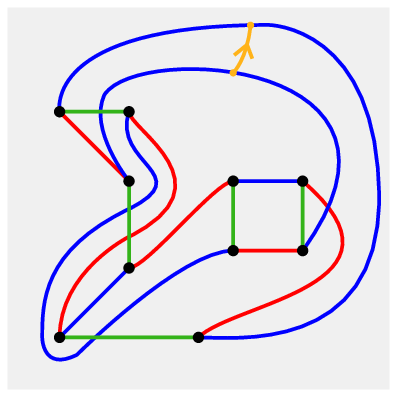}
  \label{nxb8}
\end{subfigure}
\begin{subfigure}{.25\textwidth}
  \centering
  \includegraphics[width=.9\linewidth]{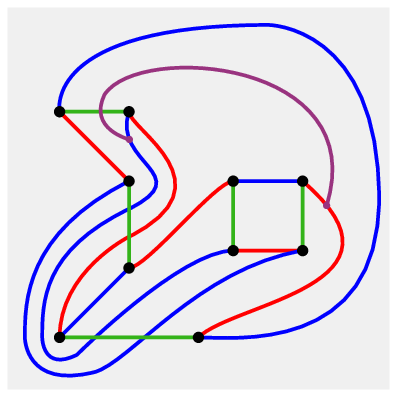}
  \label{nxb9}
\end{subfigure}%
\begin{subfigure}{.25\textwidth}
  \centering
  \includegraphics[width=.9\linewidth]{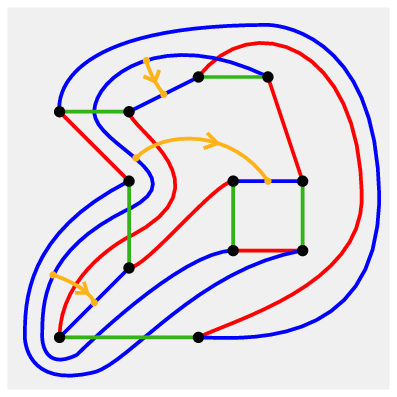}
  \label{nxb10}
\end{subfigure}
\begin{subfigure}{.25\textwidth}
  \centering
  \includegraphics[width=.9\linewidth]{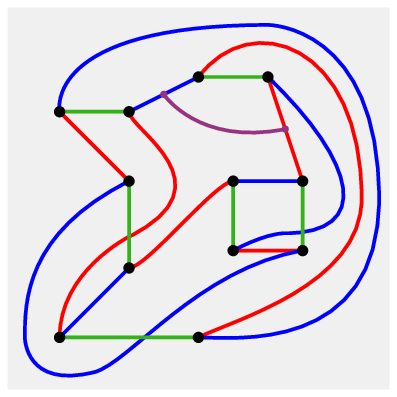}
    \label{nxb11}
\end{subfigure}%
\begin{subfigure}{.25\textwidth}
  \centering
  \includegraphics[width=.9\linewidth]{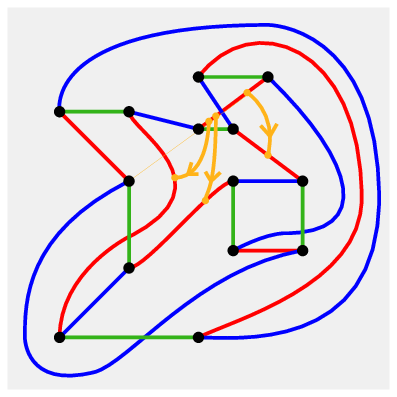}
  \label{nxb12}
\end{subfigure}
\begin{subfigure}{.25\textwidth}
  \centering
  \includegraphics[width=.9\linewidth]{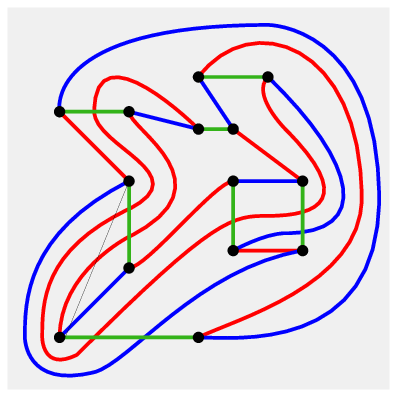}
  \label{nxb13}
\end{subfigure}
\caption{A sequence of moves, alternating between crosscap summations and shadow slides, starting with the 1-bridge trisection of the unknotted 2-sphere and ending with a bridge trisection of an unknotted nonorientable genus six surface with 1-skeleton isomorphic to $\Sigma_N$.}
\label{nxb}
\end{figure}

Finally, we convert the two diagrams with 1-skeleta $\Sigma_O$ and $\Sigma_N$ via diffeomorphisms to diagrams in which the red and blue arcs form a regular 14-gon.  The final results are shown in Figure~\ref{final}.

\begin{figure}[h!]
\begin{subfigure}{.5\textwidth}
  \centering
  \includegraphics[width=.7\linewidth]{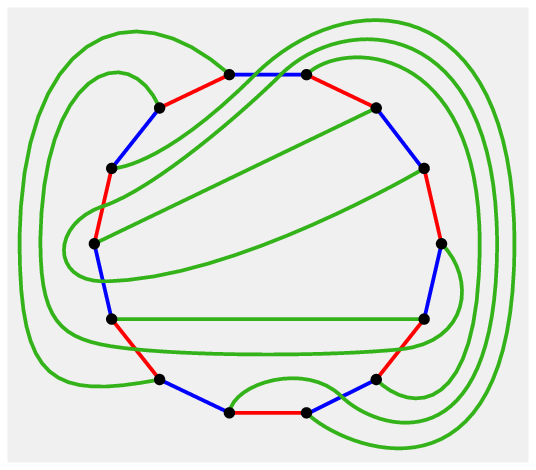}
  \label{ofinal}
\end{subfigure}%
\begin{subfigure}{.5\textwidth}
  \centering
  \includegraphics[width=.7\linewidth]{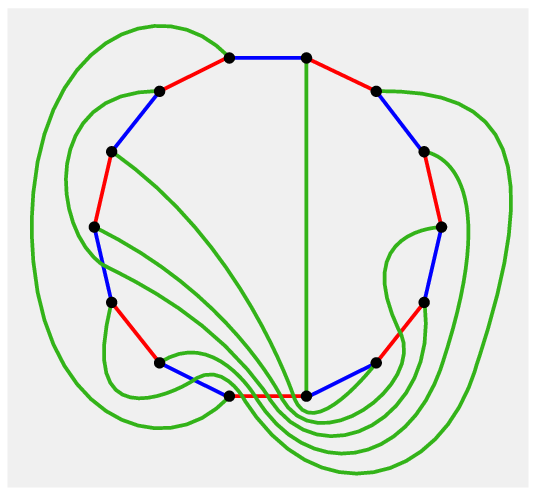}
  \label{nfinal}
\end{subfigure}
\caption{Versions of shadow diagrams of bridge trisections with 1-skeleta isomorphic to $\Sigma_O$ (left) and $\Sigma_N$ (right) in which the red-blue curve is drawn as a regular 14-gon.}
\label{final}
\end{figure}

\bibliographystyle{amsalpha}
\bibliography{cubic}

\end{document}